\documentclass{amsart}
\usepackage{amsmath}
\usepackage{amssymb} 
\usepackage{graphicx} 

\newtheorem{theorem}{Theorem}[section]

\newtheorem{lemma}{Lemma}[section]

\numberwithin{equation}{section}

  \def\b#1{\mathbf{#1}} 
\def\a#1{\begin{align*}#1\end{align*}} \def\an#1{\begin{align}#1\end{align}} \def\t#1{\hbox{#1}}

\begin{document} 

\title[finite elements]{ The nodal basis of $C^m$-$P_{k}^{(3)}$ and $C^m$-$P_{k}^{(4)}$
    finite elements on tetrahedral and 4D simplicial grids
}

 \author {Shangyou Zhang}
\address{Department of Mathematical Sciences, University of Delaware,
    Newark, DE 19716, USA. }
\email{szhang@udel.edu}

\begin{abstract}
We construct the nodal basis of $C^m$-$P_{k}^{(3)}$ ($k \ge 2^3m+1$)
    and $C^m$-$P_{k}^{(4)}$ ($k \ge 2^4m+1$)
    finite elements on 3D tetrahedral and 4D simplicial grids, respectively.
$C^m$-$P_{k}^{(n)}$ stands for the space of globally $C^m$ ($m\ge1$) 
   and locally piecewise $n$-dimensional polynomials of
   degree $k$ on $n$-dimensional simplicial grids. 
We prove the uni-solvency and the $C^m$ continuity of the 
   constructed $C^m$-$P_{k}^{(3)}$ and $C^m$-$P_{k}^{(4)}$ finite element spaces.
A computer code is provided which generates the index set for the nodal basis 
   of $C^m$-$P_k^{(n)}$ finite elements on $n$-dimensional simplicial grids. 

\end{abstract}



\maketitle

 \baselineskip=14pt

\section{Introduction} 
The Argyris finite element is one of the first finite elements, cf. \cite{Ciarlet}.
It is a $C^1$-$P_5^{(2)}$ finite element on triangular grids.  
Here $C^1$-$P_5^{(2)}$ denotes the space of globally $C^1$ and locally piecewise polynomials of
   degree 5 on 2 dimensional triangular grids. 
In general  $C^m$-$P_{k}^{(n)}$ stands for the space of globally $C^m$ ($m\ge1$) and locally piecewise $n$-dimensional polynomials of
   degree $k$ on $n$-dimensional simplicial grids. 
It is straightforward to extend the Argyris finite element to $C^1$-$P_{k}^{(2)}$ ($k> 5$) finite elements, as follows.
We introduce $(k-5)$ function values at $(k-5)$ internal points on each edge; we 
    introduce (additional) first-order normal derivatives at
            $(k-4)$ internal points each edge;
  and we introduce additional
           function values at $\dim P_{k-6}^{(2)}$ internal points on the triangle.

In 1970, Bramble and Zl\'{a}mal \cite{Bramble}   and  \v{Z}en\'i\v{s}ek \cite{Zenisek}
  extended the above $C^1$-$P_{k}^{(2)}$ finite element to $C^m$-$P_{4m+1}^{(2)}$ finite elements
  for all $m\ge 1$.
In fact, \v{Z}en\'i\v{s}ek \cite{Zenisek} defined all $C^m$-$P_{k}^{(2)}$ finite elements for $k\ge 4m+1$.
We have a perfect partition of index in two space-dimensions.
That is,  \begin{itemize}
  \item we require $2m$th order continuity at each vertex of the triangulation, and the degrees of freedom 
   are exactly the function value, the two first derivatives, and up to the $2m+1$ $2m$th derivatives at
     each vertex; 
  \item we require $m$th order continuity on each edge of the triangulation, and the degrees of freedom 
   are exactly the function value at $(k-4m-1)$ internal points, the first normal derivatives at
     $(k-4m)$ internal points, and up to the $m$th normal derivatives at $(k-3m-1)$ internal points
     inside each edge; 
   \item the degrees of freedom are exactly the function value at $\dim P_{k-3m-3}^{(2)}$ internal points 
     inside each triangle.
\end{itemize} 
For $C^m$-$P_{4m+2}^{(2)}$ finite elements, the above sets of index form exactly seven triangles, three at three 
   vertices, three at three edges, and one at the center of the triangle.
Such a nice partition of index does not exists in three and higher dimensions.

The first $C^1$ element in 3D was constructed
     by  \v{Z}en\'i\v{s}ek in 1973 \cite{Zenisek-3d}, a $C^1$-$P_9^{(3)}$
   finite element.
By avoiding high-order derivatives in the degrees of freedom of \v{Z}en\'i\v{s}ek  \cite{Zenisek-3d},
   the author extended this finite element to all $C^1$-$P_k^{(3)}$  ($k\ge 9$) finite elements \cite{Zhang-3d}
   in 2009.
Also \v{Z}en\'i\v{s}ek  extended the $C^1$-$P_9^{(3)}$ finite element to $C^m$-$P_{8m+1}^{(3)}$ 
   in 1974 \cite{Zenisek-cm}.
Again high-order derivatives (above the continuity order) were used as degrees of freedom in \cite{Zenisek-cm}.
The author could not extend the 3D $C^m$ \v{Z}en\'i\v{s}ek finite element to all $C^m$-$P_{k}^{(3)}$
   ($k\ge 8m+1$) finite elements, but constructed a family of $C^2$-$P_{k}^{(3)}$
   ($k\ge 17$) finite elements using the continuity-equal order derivatives in \cite{Zhang-4d}, in 2016.
Also the author defined a family of $C^1$-$P_{k}^{(4)}$ ($k\ge 17$) finite elements  in \cite{Zhang-4d}.

For $C^m$-$P_{k}^{(n)}$ ($k\ge 2^n m+1$) finite elements on $n$ ($n\ge 3$) space-dimensional simplicial grids,
  by the author's limited knowledge,  Alfeld, Schumaker and Sirvent are the first to
   introduce the distance concept
     to low-dimensional simplex and to define recursively the index for the nodal basis 
   (degrees of freedom),  cf.  Equation (36) in \cite{Alfeld}.
As claimed by \cite{Alfeld}, it is very difficult to find explicit definitions of these index sets for general or 
   given  $m$, $n$ and $k$.

Similar to \cite{Alfeld}, \cite{Hu} and \cite{Chen} studied the index partition recently.
But no closed formula is obtained for the index sets for general $C^m$-$P_{k}^{(n)}$ finite elements.
For example,  \cite{Hu} obtained the index sets (degrees of freedom) for 
   the $C^4$-$P_{33}^{(3)}$ finite elements on tetrahedral grids.

Mathematically it is a challenge to find explicit definitions of basis functions for 
    general $C^m$-$P_{k}^{(n)}$ finite elements.
For so many years we could not even complete the work in 3D, except for $m=1$ (\cite{Zhang-3d})
   and $m=2$ (\cite{Zhang-4d}). 
In this work,  we give explicitly the index set 
    of the nodal basis for all 
   $C^m$-$P_{k}^{(3)}$ ($k\ge 2^3m+1$) and $C^m$-$P_{k}^{(4)}$ ($k\ge 2^4m+1$)  finite elements,
    on tetrahedral grids and four-dimensional simplicial grids, respectively.
We did not use any formula of \cite{Alfeld}, \cite{Hu}, or \cite{Chen}, in giving the
 closed formula on the index of  the  nodal bases of 
   $C^m$-$P_{k}^{(3)}$ and $C^m$-$P_{k}^{(4)}$ finite elements. 
We prove the uni-solvency and the $C^m$ continuity of the 
   constructed $C^m$-$P_{k}^{(3)}$ and $C^m$-$P_{k}^{(4)}$ finite element spaces.

We also provide a computer code which produces the index sets for any given $n$, $m$ and $k$.
The computer code is listed in Section \ref{code}.
The computer code follows the 
   continuity requirements and goes exhaustively and non-overlappingly from nodal indices on lower
   dimensional simplicial faces to those on higher ones. 
The computer code does not solve this mathematical problem, but does give explicit 
  indices for practical computation need.
With the computer output,  we study the patterns of overlapping index so that we
  can give explicit nodal basis definitions in 3D and 4D.
The computer code verifies the index sets, for some low $n$'s, $m$'s and $k$'s, 
   constructed in this manuscript.
In particular,  it verifies the indices of the $C^4$-$P_{33}^{(3)}$ 
   finite element in \cite{Hu}.

\section{The 3D $C^m$-$P_k^{(3)}$ finite elements}

Let $\mathcal{T}_h=\{K\}$ be a 3D tetrahedral grid where a tetrahedron 
   $K$ has 4 vertices $\{\b x_i=
   (x_1^{(i)}, x_2^{(i)}, x_3^{(i)}), \ i=0,\dots,3\}$, the intersection of
   of two $K$'s is either a common  face-triangle, or a
   common edge, or a common vertex, or an empty set, and the maximal diameter of $\{K\}$ is $h$.

\long\def\mskip#1{}

\mskip{ 
restart; n:=4: m:=1: k_1:=1: k:=m*2^n+1+k_1: dm:=(k+1)*(k+2)*(k+3)*(k+4)/24;
for i from 0 to n do m1[i]:=m*2^(n-1-i): od: m1[n]:=0:
for i from 0 to n do sc[i]:=vector(1+m*2^(n-1),0): 
  for j to n+1 do st[i,j]:=matrix(dm,n+1,0): od: od:
for i from 0 to 0 do for j from 0 to m1[i] do for i0 from k to 0 by -1 do 
 for i1 from k-i0 to 0 by -1 do               for i2 from k-i0-i1 to 0 by -1 do 
   for i3 from k-i0-i1-i2 to 0 by -1 do       for i4 from k-i0-i1-i2-i3 to 0 by -1 do  
    if i0=k-j then sc[i][j+1]:=sc[i][j+1]+1: s:=sc[i][j+1]:
      t[s,1]:=i0: t[s,2]:=i1: t[s,3]:=i2: t[s,4]:=i3: t[s,5]:=i4: fi:
    od:od:od:od:od: st[i,j]:=matrix(sc[i][j+1],n+1,0): 
    for i0 to sc[i][j+1] do for j0 to n+1 do st[i,j][i0,j0]:=t[i0,j0]: od: od:  od: od:

i:=0:  evalm(sc[i]); #for j from 0 to m1[i] do j,evalm(st[i,j]); od;

}

\mskip{ 
 m:=2: k1:=0:  k:=8*m+1+k1;  n:=3:  s1:=4*(4*m+1)*(4*m+2)*(4*m+3)/6;
   j:=0: for s from 0 to 2*m do j:=j+(s+k1)*(s+1); od; j;
s2:=6*( ((2*m + 1)^2+(2*m + 1))*k1/2 + ((2*m + 1)^3-(2*m+1))/3 );
s3:=4*((m + 1)*(3*k1^2 + 18*k1*m + 25*m^2 - 3*k1 - 4*m))/6;
s4:= (4*m+k1-2)*( 4*m+k1-1)*( 4*m+k1)/6- 4*(m-2)*(m-1)*m/6;
s1+s2+s3+s4; ss:=(8*m+k1+2 )* (8*m+k1+3 )* (8*m+k1+4 )/6;

restart;  s_vertex:=(4*m+1)*(4*m+2)*(4*m+3)/6:
s_edge:=sum((1+i)*(k1+i),i=0..2*m);
s_face:=((m + 1)*(3*k1^2 + 18*k1*m + 25*m^2 - 3*k1 - 4*m))/6:
s_tet:= simplify(sum((2*m - 1 + k1 + 2*k)*(2*m + k1 + 2*k)/2, k = 0 .. m) - (m-1)*m*(m+1)/2);

s1:=4*(4*m+1)*(4*m+2)*(4*m+3)/6; 
s2:=6*( ((2*m + 1)^2+(2*m + 1))*k1/2 + ((2*m + 1)^3-(2*m+1))/3 );
s3:=4*((m + 1)*(3*k1^2 + 18*k1*m + 25*m^2 - 3*k1 - 4*m))/6;
s4:= (4*m+k1-2)*( 4*m+k1-1)*( 4*m+k1)/6- 4*(m-2)*(m-1)*m/6;
expand(s1+s2+s3+s4-(8*m+k1+2 )* (8*m+k1+3 )* (8*m+k1+4 )/6); 

}
 
For 3D $C^m$-$P_k$ ($k\ge 2^3m+1$) finite elements,  we have 4 types of nodal basis functions.

\begin{description}
\item[\bf 1] At 4 vertices,  the function value, the three first derivatives and up to the $\dim P_{4m}^{(2)}$   
   $4m$th order derivatives, are selected.  There are 
 \an{\label{3-v}  \sum_{i=0}^{4m} \dim P_i^{(2)} = \dim P_{4m}^{(3)}  = \frac {  (4m+1)(4m+2)(4m+3)} 6
} degrees of freedom at each vertex.

\item[\bf 2] At 6 edges,  the function values at  $k_1$ internal points ($k_1=k-2^3 m-1$)  and the $2$ first derivatives at $k_1+1$ internal points, and up to $2m+1$  $ 2m$th
  order normal derivatives (normal
   to the edge) at $(k_1+2m)$ internal points,   are selected. On one edge there are 
 \an{\label{3-e} &\quad \ \sum_{i=0}^{2m} (1+i) \dim P_{k_1+i-1}^{(1)} \\
   \nonumber & = k_1 \frac { (2m+1)^2+(2m+1) }2 + \frac { (2m+1)^3- (2m+1) } 3
} degrees of freedom.

\item[\bf 3] On 4 face-triangles,  the function values at  $\dim P_{k-6m-3}^{(2)}$ internal points, and the first normal derivatives at
    $\dim P_{k-6m-1}^{(2)}$ internal points, and up to order $ m$ normal derivatives at 
     $\dim P_{k-4m-3}^{(2)}$ internal points, 
    except near each of 3 corners of a triangle $\dim P_{m-2}^{(3)}$ order $2$ to order $m$ normal derivatives,  are selected. 
  That is, we drop one order $2$ normal derivative at $\dim P_{0}^{(2)}$ internal corner point(s),
                         the order $3$ normal derivative at $\dim P_{1}^{(2)}$ internal corner points,
    and up to the order $m$  normal derivative at $\dim P_{m-2}^{(2)}$ internal corner points.
   On one face-triangle,  there are 
 \an{\label{3-f}   &\quad \ \sum_{i=0}^{m} \dim P_{k-6m-3+2i}^{(2)}-3 \dim P_{m-2}^{(3)}\\ \nonumber &= 
           \sum_{i=0}^{m} \frac{(2m-1+k_1 +2i)(2m+k_1+2i)} 2 -\frac{ (m-1)m(m+1)} 2 \\
  \nonumber &= (m + 1)\frac{ 3 k_1^2 + 3k_1 (6 m  -1)  + 25 m^2- 4 m}6
} degrees of freedom.

\item[\bf 4] Inside the tetrahedron,  the $\dim P_{k-4m-4}^{(3)}$ function values, except near
       each of four corners $\dim P_{m-3}^{(3)}$ internal function values, are selected.
    There are 
 \an{\label{3-tet} &\quad \ \dim P_{k-4m-4}^{(3)} -4 \dim P_{m-3}^{(3)}
   \\ \nonumber &=  \frac{( 4m+k_1-2 )( 4m+k_1-1)( 4m+k_1)} 6-\frac{4(m-2)(m-1)m}6
} degrees of freedom.

\end{description}

Adding all degrees of freedom in \eqref{3-v}--\eqref{3-tet},  we obtain
\a{
    \t{dof}_3 &= 4 \cdot  \frac {  (4m+1)(4m+2)(4m+3)} 6 \\
           &\quad \ + 6 \cdot (k_1 \frac { (2m+1)^2+(2m+1) }2 + \frac { (2m+1)^3-2m-1 } 3) \\
           &\quad \ + 4 \cdot  (m + 1)\frac{ 3 k_1^2 + 3k_1 (6 m  -1)  + 25 m^2- 4 m}6\\
           &\quad \ +   \frac{( 4m+k_1-2 )( 4m+k_1-1)( 4m+k_1)} 6\\
           &\qquad -\frac{4(m-2)(m-1)m}6\\
         &=  \frac{(8m+k_1+2 )( 8m+k_1+3)(8m+k_1+4)} 6. }
\an{\label{3-m}
 \dim  P_{k }^{(3)} &=  \frac{(8m+k_1+2 )( 8m+k_1+3)(8m+k_1+4)} 6=\t{dof}_3.\qquad\qquad  }
    The two numbers match.

We next limit ourselves to the unit right tetrahedron
\a{ K=\{ (x_1,x_2,x_3) : 0\le x_1, x_2, x_3, x_1+x_2+x_3 \le 1 \}.  }
For a general tetrahedron $K=(\lambda_1,\lambda_2,\lambda_3)$, the only change is 
   the coordinate system, using the barycentric coordinates instead of $(x_1,x_2,x_3)$.
The restriction of $K$ by $x_3=0$ is the 2D unit right triangle
\a{ T=\{ (x_1,x_2) : 0\le x_1, x_2,  x_1+x_2 \le 1 \}.  }

\begin{lemma}\label{2l}
  The subset of linear functionals in \eqref{3-v}--\eqref{3-f}, when restricted on
  the face-triangle $T$, uniquely determines a 2D polynomial $p_k\in P^{(2)}_k(T)$.
\end{lemma}
\begin{proof} The restricted sub-subset of linear functionals is also described in the Introduction of
  this manuscript, for the $C^{2m}$-$P_k^{(2)}$ (not $C^m$) 
  finite element.  The lemma is proved by \cite{Bramble}.
\end{proof}

\begin{theorem}\label{3u} The set of linear functionals in \eqref{3-v}--\eqref{3-tet}
    uniquely determines a 3D polynomial $p_k\in P^{(3)}_k(K)$.
\end{theorem}
\begin{proof} By \eqref{3-m},  we have a square linear system of finite equations so that 
  we only need to prove the uniqueness of the solution.   Let $p_k\in P^{(3)}_k(K)$ having
  all dof's in \eqref{3-v}--\eqref{3-tet} zero.  

By Lemma \ref{2l} and \eqref{3-v}--\eqref{3-f},
\a{ p_k = p_{k-1} x_3, \ \t{for some } \ p_{k-1}\in P^{(3)}_{k-1}(K). }
By \eqref{3-e}, $p_{k-1}$ vanishes at all three vertices of $T$:
\an{\label{v0} p_{k-1}(\b v_i) =0, \ \b v_i\in T, i=0,\dots,2.  }
Using the first normal derivative and all directional derivatives $\partial_{1,2,3}^{m_1,m_2,1}
   p_{k}$ ($= \partial_{1,2,3}^{m_1,m_2,0}
   p_{k-1}$ on $T$)
   in \eqref{3-v}--\eqref{3-f},  all degrees of freedom of $C^{2m-1}$-$P_{k-1}^{(2)}$
   of $p_{k-1}|_T$ vanish.  By Lemma \ref{2l}, we get $p_{k-1}|_T =0$ and 
\a{ p_{k-1} = p_{k-2} x_3, \ \t{for some } \ p_{k-2}\in P^{(3)}_{k-2}(K). }
Doing this $m-1$ times more, and on the other three face triangles,  we get
\an{\label{B} p_k = p_{k-4m-4} B, \ \t{for some } \ p_{k-4m-4}\in P^{(3)}_{k-4m-4}(K), }
where $B\in P_{4m+4}$ has its zeroth to $m$-th normal derivatives vanished on the
   4 face-triangles of $K$. 

By \eqref{3-v}--\eqref{3-f},
\an{\label{v30} \partial_{1,2,3}^{m_1,m_2,m_3} p_{k-4m-4}(\b v_i)
    =0, \ 0\le m_1+m_2+m_3\le (m-3), \ i=0,\dots,3, }
where $\{\b v_i\}$ are 4 vertices of $K$. We note that, in \eqref{v30},
    the high order derivatives on $p_k$ from \eqref{3-v}-\eqref{3-f} become low order
   derivatives on $p_{k-4m-4}$.
By \eqref{3-tet} and \eqref{v30},
\a{ p_{k-4m-4}=0. }
The proof is complete.
\end{proof}

\begin{theorem}\label{3c} The finite element space
\a{ V_h=\Big\{v \in L^2(\Omega) : v|_K=\sum_{i=1}^{\operatorname{dof}_3} f_i(v) \phi_i, \quad
   K\in \mathcal{T}_h \Big\} \subset C^m(\Omega), }
where linear functionals $\{f_i\}$ are defined in \eqref{3-v}--\eqref{3-tet},
  and $\{\phi_i\}$ is the dual basis of $\{f_i\}$ on $K$.
\end{theorem}
\begin{proof} Let $F$ be a common face-triangle of $K_1$ and $K_2$. Let $\lambda_1=0$ be
   a linear equation for the plane on which $F$ is. 
   By \eqref{B},   
\a{ v_1-v_2 = \lambda_1^{m+1} p_{k-m-1}, \ \t{for some } \ p_{k-m-1}\in P^{(3)}_{k-m-1}(\Omega), }
where $v_1$ and $v_2$ are global polynomials whose restrictions are
   $v|_{K_1}$ and $v|_{K_2}$, respectively.
Therefore the function value and all the normal derivatives up to $m$-th order, on $F$, are
   zero.  Thus $v$ is $C^m$ on face $F$.  The continuity on edges and at vertices are
   direct corollaries of face-continuity. 
The proof is complete.
\end{proof}

\section{The 4D $C^m$-$P_k^{(4)}$ finite elements}

Let $\mathcal{T}_h=\{S\}$ be a 4D simplicial grid where $S$ has 5 vertices $\{\b x_i=
   (x_1^{(i)}, x_2^{(i)}, x_3^{(i)},$ $ x_4^{(i)}), \ i=0,\dots,4\}$, the intersection of
   of two $S$'s is either a common face-tetrahedron, or a common face-triangle, or a
   common edge, or a common vertex, or an empty set, and the maximal diameter of $\{S\}$ is $h$.
\long\def\mskip#1{}
\mskip{ 
 for k from 1 to 5 do k,(k+1)*(k+2)*(k+3)*(k+4)/24; od;

 k_1:=0:  n:=4:  s4dof:=0:
m:=5: k:=2^n*m+1+k_1; d0:=m*8: s1:=(n+1)*(d0+1)*(d0+2)*(d0+3)*(d0+4)/24:
   j:=0: for s from 0 to 4*m do j:=j+(s+k_1)*(s+1)*(s+2)/2; od: s2:=j;
   j:=0: for i from 0 to 2*m do j:=j+(i+1)*(k-12*m-3+2*i+1)*(k-12*m-3+2*i+2)/2- 3* (i+1)*(i-2+1)*(i-2+2)/2; od:
    s3:=j;
   for i from 0 to m do t3[i]:=(k-8*m-4+3*i+1)*(k-8*m-4+3*i+2)*(k-8*m-4+3*i+3)/6
       -4*((2*m+2*i-2)*(2*m+2*i-1)*(2*m+2*i-0)/6) ; od: t4[0]:=0: t4[1]:=0:
   for i from 2 to m do t4[i]:=0: for j2 from 2 to i do t4[i]:=t4[i]+6*(i-j2+1)*(4*m+2-j2+k_1); od; od:
   for i from 0 to m do s3dof=i, t3[i]-t4[i]; od;
   s_tet:=0: for i from 0 to m do s_tet:=s_tet + t3[i]-t4[i]; od: 

    d0:=11*m-4+k_1: in1:=(d0+1)*(d0+2)*(d0+3)*(d0+4)/24:
    d0:=4*m-4: in2:=(d0+1)*(d0+2)*(d0+3)*(d0+4)/24:
               in3:=0: for i from 3 to m do in3:=in3+(m-i+1)*(m-i+2)/2*(4*m+k_1-i+3): od: 
   s4s:=in1-5*in2-10*in3:  
   s4dof:=s4dof,[s4s,k]; d0:=k: pks:=(d0+1)*(d0+2)*(d0+3)*(d0+4)/24:
   [pks,  s1+10*s2+10*s3+5*s_tet+s4s];

v0d:=[[325,17],[6965,33],[38325,49],[126021,65]]; s4dof;
v0d:=[[8505,34],[43771,50],[139225,66]];
v0d:=[[153375,67]];

restart;   d0:=m*8:  s_vertex:=(d0+1)*(d0+2)*(d0+3)*(d0+4)/24; 
s_edge1:=sum((1+i)*(i+2)/2*(k_1+i),i=0..4*m):
s_edge:=k_1*(4*m+1)*(4*m+2)*(4*m+3)/6+(4*m )*(4*m+1)*(4*m+2)*(4*m+3)/8; expand(s_edge1-s_edge);
s_face1:=simplify(subs(k=16*m+1+k_1, sum((i+1)*(k-12*m-3+2*i+1)*(k-12*m-3+2*i+2)/2- 3* (i+1)*(i-2+1)*(i-2+2)/2, i=0..2*m))):
s_face:=(m+1)*(2*m+1)*(3*k_1^2+40*k_1*m+118*m^2-3*k_1-7*m)/6; expand(s_face1-s_face);

s_tet:= simplify(sum((1+k_1+8*m-4+3*i+1)*(1+k_1+8*m-4+3*i+2)*(1+k_1+8*m-4+3*i+3)/6
           -4*((2*m+2*i-2)*(2*m+2*i-1)*(2*m+2*i-0)/6), i=0..m)
           -6*sum(sum( (i-j+1)*(4*m+2-j+k_1),j=2..i),i=2..m) );
s_tet1:=(m + 1)*( m* (2945*m^2-491*m+6)/24 + (546*m^2-105*m+4)*k_1/12  + (19*m-2)*k_1^2/4 + k_1^3/6);
    expand(s_tet1-s_tet);

simplify(s_tet-(1/24)*m*(m+1)*(2945*m^2-491*m+6)
              -(1/12*(m+1))*(546*m^2-105*m+4)*k_1
              -(1/4*(m+1))*(19*m-2)*k_1^2
              -(1/6*(m+1))*k_1^3);
factor((1/4)*m-(485/24)*m^2+(409/4)*m^3+(2945/24)*m^4);
factor((1/3)*k_1-(101/12)*m*k_1+(147/4)*k_1*m^2+(91/2)*k_1*m^3);
factor(     -(1/2)*k_1^2+(17/4)*k_1^2*m+(19/4)*k_1^2*m^2);
factor(   +(1/6)*k_1^3*m     +(1/6)*k_1^3);

d0:=11*m-4+k_1:  s4:=(d0+1)*(d0+2)*(d0+3)*(d0+4)/24 - 5*(4*m-3)*(4*m-2)*(4*m-1)*(4*m)/24:
factor(sum((m-i+1)*(m-i+2)/2*(4*m+k_1-i+3),i=3..m));

s4:=s4-10* ((1/24)*m*(m-1)*(m-2)*(4*k_1+15*m+3)):
pk:=(1+16*m+1+k_1)*(2+16*m+1+k_1)*(3+16*m+1+k_1)*(4+16*m+1+k_1)/24:
diffs:=factor(5*s_vertex+10*s_edge+10*s_face+5*s_tet+s4-pk); 

subs({m=4,k_1=1},[16*m+1+k_1, s_vertex, s_edge, s_face, s_tet, s4, pk,diffs]);

}
 
For 4D $C^m$-$P_k$ ($k\ge 2^4m+1$) finite elements,  we have 5 types of nodal basis functions.

\begin{description}
\item[\bf 1] At 5 vertices,  the function value and the derivatives up to order $ 2^3m$ are selected.  There are 
 \an{\label{4-v}  \operatorname{dof}_{4,0}=
           \dim P_{2^3m}^{(4)}  = \frac {  (8m+1)(8m+2)(8m+3)(8m+4)} {24}
} degrees of freedom at each vertex.

\item[\bf 2] At 10 edges,  the $k_1$ function values, ($k_1=k-2^4 m-1$) 
   and the $\dim P_{1}^{(2)}$ first derivatives at $k_1+1$ points, and up to 
  order $4m$ $\dim P_{4m}^{(2)}$ normal derivatives at $k_1+4m$ points,  are selected. 
  Inside one edge there are 
 \an{\label{4-e} &\quad \ \operatorname{dof}_{4,1}= \sum_{i=0}^{4m} (k_1+i) \dim P_{i}^{(2)} \\
   \nonumber & = k_1 \frac {m_1(m_1+1)(m_1+2)}6 + \frac {(m_1-1)m_1(m_1+1)(m_1+2) }8
} degrees of freedom,  where $m_1=4m+1$.

\item[\bf 3] On each of 10 face-triangles, with two normal vectors on each triangle, we are supposed to select 
     the function values at  $\dim P_{k-12m-3}^{(2)}$ points , and the $\dim P_{1}^{(1)} $
            first derivatives at $\dim P_{k-12m-1}^{(2)}$  points, and up to  order-$2m$ $\dim P_{2m}^{(1)}$
           normal derivatives at $\dim P_{k-8m-3}^{(2)}$ points. 
   However due to corner overlaps,  since the second derivative,
   we drop $\dim P_2^{(1)}$ third normal derivatives at $\dim P_0^{(2)}$ points at each of three corners of the triangle,
               $\dim P_3^{(1)}$ fourth normal derivatives 
     at $\dim P_1^{(2)}$ points at each of three corners of the triangle,
   and up to $\dim P_{2m}^{(1)}$ $2m$-th normal derivatives 
     at $\dim P_{2m-2}^{(2)}$ points at each of three corners of the triangle.
    On one face-triangle,  there are 
 \an{\label{4-f}   &\quad \ \operatorname{dof}_{4,2}=\sum_{i=0}^{2m} \Big(\dim P_{i}^{(1)}  \dim P_{k-12m-3+2i}^{(2)} -
                                 3\dim P_{i}^{(1)}   \dim P_{i-2}^{(2)} \Big)  \\ 
  \nonumber &= (m + 1)(2m + 1)\frac{ 3 k_1^2 + 40 k_1 m + 118 m^2 - 3 k_1 - 7 m }6
} degrees of freedom inside each triangle.

\item[\bf 4] On each of 5 face-tetrahedra of the 4D simplex,  we are supposed to select 
    the function value at  $\dim P_{k-8m-4}^{(3)}$ internal points, the first normal derivative at
    $\dim P_{k-8m-1}^{(3)}$  internal points, and up to the $m$-th normal derivative at
    $\dim P_{k-5m-3}^{(3)}$   internal points.  For the first derivative and higher normal derivatives,
   we have vertex overlapping and we drop near each of 4 corners of a tetrahedron,
     the $i$-th norm derivative ($0\le i\le m$) at $\dim P_{m+2i-2}^{(3)}$ internal points.
  The overlap with edge degrees of freedom is complicated.  The overlapped indices form 
   a  wedge with two end-triangles non-parallel.  
    That is,  along each edge,  we drop second normal derivatives at  $4m$ internal (tetrahedron) points,
   third normal derivatives at $2(4m)+(4m-1)$ points,  fourth normal derivatives at $3(4m)+2(4m-1)+(4m-2)$
   points, and so on until $m$-th normal derivatives.
    There are 
 \an{\label{4-t}  &\quad \ \operatorname{dof}_{4,3}=\sum_{i=0}^{m} \Big( \dim P_{k-8m-4+3i}^{(3)} 
               -      4  \dim P_{i-1}^{(3)} \Big)   \\ 
  \nonumber &\qquad - 6 \sum_{i=2}^{m}\sum_{j=2}^i (i-j+1)(4m+2-j)   \\ 
  \nonumber &= (m + 1)\Big( \frac{ m (2945m^2-491m+6)}{24}  \\ 
  \nonumber &\qquad + \frac{(546m^2-105m+4)k_1}{12}  + \frac{(19m-2)k_1^2} 4 + \frac{k_1^3}6 \Big)
} degrees of freedom inside each tetrahedron.

\item[\bf 5] Inside the 4D simplex,  the function values at $\dim P_{k-5m-5}^{(4)}$ internal points are supposedly 
   selected, but near each of five corners we 
 drop $\dim P_{4m-4}^{(4)}$ internal function values,
   and for $m\ge 3$ near each of ten edges we 
   drop the function value at internal points which 
   form 4-dimensional wedges. 
   The barycentric coordinates of these 4-dimensional wedge points, near
   the edge $\b x_4\b x_5$,
        have the form $(c_1,c_2,c_3,c_4,c_5)$ with $\max\{c_4,c_5\} \le 2^3m+k_1$, 
            $c_4+c_5=(2^2+ 2^3)m+k_1+(i-3)$ where $i=3,\dots, m$,
      and $c_1+c_2+c_3\le 4m-(i-3)$. 
   There are, inside the 4D simplex, 
 \an{\label{4-s}  \operatorname{dof}_{4,4}&=\dim P_{k-5m-5}^{(4)} -5 \dim P_{4m-4}^{(4)}
       \\
         \nonumber &\quad \ - 10\sum_{i=3}^m \dim_{m-i}^{(2)} (4m+k_1-i+3)  \\
      \nonumber &=  \frac{(11 m-3+k_1)(11 m-2+k_1)(11 m-1+k_1)(11 m+k_1)} {24} \\
         \nonumber &\quad \ -\frac{5(4m-3)(4m-2)(4m-1)(4m)} {24} \\
         \nonumber &\quad \                -\frac{10(m-2)(m-1)m (4k_1+15m+3)}{24} 
} degrees of freedom.

\end{description}
Adding all degrees of freedom in \eqref{4-v}--\eqref{4-s},  we obtain the total degrees of freedom on one
   4D simplex, 
\a{
    \t{dof}_4 &=5 \operatorname{dof}_{4,0} 
               +10 \operatorname{dof}_{4,1}
               +10 \operatorname{dof}_{4,2}
                +5 \operatorname{dof}_{4,3}+\operatorname{dof}_{4,4}  \\
         &=  \frac{(16m+k_1+2)(16m+k_1+3)(16m+k_1+4)(16m+k_1+5)}{24}  \\
         &=  \frac{(k+1)(k+2)(k+3)(k+4)}{24}. }
\an{\label{4-m} 
        \dim  P_{k }^{(4)} 
     &=\frac{(k+1)(k+2)(k+3)(k+4)}{24}=\t{dof}_4. \qquad\qquad\qquad\qquad
   }

\begin{theorem} The set of linear functionals in \eqref{4-v}--\eqref{4-s}
    uniquely determines a 4D polynomial $p_k\in P^{(4)}_k(S)$.
\end{theorem}
\begin{proof} By \eqref{4-m},  we have a square linear system of finite equations. 
  The uniqueness implies the existence of solution.   
  Let $p_k\in P^{(4)}_k(S)$ having
  all dof's in \eqref{4-v}--\eqref{4-s} zero.  
  By Theorem \ref{3u} and \eqref{4-v}--\eqref{4-t}, as all dof's of $C^{2m}$-$P^{(3)}_k(K)$
   of $p_k|_K$ vanish,
\a{ p_k = p_{k-1} x_4, \ \t{for some } \ p_{k-1}\in P^{(4)}_{k-1}(S), }
where we assume \a{
   S=\{(x_1,x_2,x_3,x_4) : 0\le x_1,x_2,x_3,x_4,x_1+x_2+x_3+x_4\le 1 \}. }

By \eqref{4-v} and \eqref{4-f}, $p_{k-1}$ vanishes at all four vertices of a
  face-tetrahedron $K=\{x_4=0\}$ of $S$, because all first order derivatives of
   $p_k$ vanish at those points,
\an{\label{4v0} p_{k-1}(\b v_i) =0, \ \b v_i\in K, i=0,\dots,3.  }
Using the first normal derivative and all directional derivatives 
  $\partial_{1,2,3,4}^{m_1,m_2,m_3,1}
   p_{k}$ in \eqref{4-v}--\eqref{4-t}, by Theorem \ref{3u} and \eqref{4-t}, we have
\a{ p_{k-1} = p_{k-2} x_4, \ \t{for some } \ p_{k-2}\in P^{(4)}_{k-2}(S). }

By \eqref{4-v} and \eqref{4-f}, $p_{k-2}$ and all 4(3 on face tetrahedron) first order 
   derivatives vanish at all four vertices of a
  face-tetrahedron $K=\{x_4=0\}$ of $S$, 
\an{\label{4v1} \partial_{1,2,3,4}^{m_1,m_2,m_3,0}
     p_{k-2}(\b v_i) =0, \  0\le m_1+m_2+m_3\le 1, \ i=0,\dots,3.  }
Additionally, for $m\ge 2$,  by \eqref{4-v} and \eqref{4-f}, 
   $p_{k-2}$ vanishes at $4m$ internal points on the 6 edges of $K$, 
\an{\label{4v2}  
     p_{k-2}(\b m_{i,j}) =0, \  i=1,\dots,6, \ j=1,\dots,4m.  }
By \eqref{4v1}, \eqref{4v2} and \eqref{4-t}, $p_{k-2}|_K=0$ and
\a{ p_{k} = p_{k-3} x_4^3, \ \t{for some } \ p_{k-3}\in P^{(4)}_{k-3}(S). }

Repeating this $m-2$ times and also on the other 4 face-tetrahedra of $S$,  we get
\an{\label{B4} p_k = p_{k-5m-5} B, \ \t{for some } \ p_{k-5m-5}\in P^{(4)}_{k-5m-5}(S), }
where $B\in P_{5m+5}(S)$ is a bubble polynomial having
   its zeroth to $m$-th normal derivatives vanishing on the
   5 face-tetrahedra of $S$. 

By \eqref{4-v}--\eqref{4-t}, we have
\an{\label{4t1}  \partial_{1,2,3,4}^{m_1,m_2,m_3,m_4}
     p_{k-5m-5}(\b v_i) =0, \  0\le \sum_{i=1}^4 m_i\le m-1, \ i=0,\dots,4,  }
where $\{\b v_i\}$ are the 5 vertices of $S$.
For $m\ge 3$, by \eqref{4-v}--\eqref{4-t}, we have
\an{\label{4t2}  
     p_{k-5m-5}(\b m_{i,j}) =0, \  j=1,\dots,k-12m, \ i=1,\dots,10,  }
where $\{\b m_{i,j}\}$ are 1D Lagrange points inside the 10 edges of $S$.
Further, for $m\ge 4$,  by \eqref{4-v}--\eqref{4-t}, we have
\an{\label{4t3}  
    \partial_{\b n_{i,1},\b n_{i,2},\b n_{i,3}}^{m_1,m_2,m_3} 
    p_{k-5m-5}(\b m_{i,j}) =0, \  & m_1+m_2+m_3=1,    \\
           \nonumber & j=1,\dots,k-12m-1,\ i=1,\dots,10,  }
where $\{\b n_{i,j}\}$ are three unit normal vector on an edge $E_i$ of
    $S$, and  $\{\b m_{i,j}\}$ are 1D Lagrange points inside the 10 edges of $S$.
We note that each such first order derivative can replace a function value, in \eqref{4-s}, at a
  node internal to a face-triangle (of $S$) which has this edge as one of its three edges.
The Lagrange points in \eqref{4t2} and \eqref{4t3} are different as they belong to
   different degree polynomials.
Combining \eqref{4t2}, \eqref{4t3} and equations for higher order derivatives,  we have
\an{\label{4t4}  
    \partial_{\b n_{i,1},\b n_{i,2},\b n_{i,3}}^{m_1,m_2,m_3} 
    p_{k-5m-5}(\b m_{i,j}) =0, \  & m_0=\sum_{i=l}^3 m_l=0,\dots,m-3,    \\
           \nonumber & j=1,\dots,k-12m-m_0,\ i=1,\dots,10.  }
By \eqref{4t1}, \eqref{4t4}   and \eqref{4-s}, we conclude with $p_{k-5m-5}=0$.
The proof is complete.
\end{proof}

\begin{theorem} The finite element space
\a{ V_h=\Big\{v \in L^2(\Omega) : v|_S=\sum_{i=1}^{\operatorname{dof}_4} F_i(v) \Phi_i, \quad
   S\in \mathcal{T}_h \Big\} \subset C^m(\Omega), }
where linear functionals $\{F_i\}$ are defined in \eqref{4-v}--\eqref{4-s},
  and $\{\Phi_i\}$ is the dual basis of $\{F_i\}$ on $S$.
\end{theorem}
\begin{proof} Using \eqref{B4} instead of \eqref{B}, 
   the proof is identical to that of Theorem \ref{3c}.
\end{proof}

\section{A code computing degrees of freedom of $C^m$-$P_k^{(n)}$ } 

A computer code in Fortran is listed in Section \ref{code}.

\begin{enumerate}
 \item The code uses $n=3$, $m=3$ and $k_1=2$ to compute the nodal basis of the
    $C^m$-$P_{2^n m +1 +k_1}^{(n)}$ finite element.  One can change \verb+c0+ for
    computing other cases.
  \item The output in Section \ref{code} lists the barycentric indices of the
     first function value, the first first derivative, the first second derivative and the 
    first third derivative at a face triangle in top four lines, respectively.
     This is for studying the index.  It can be commented out by \verb+c6+.
  \item Also by changing \verb+c6+ (and the \verb+if+ statement above it) we can 
     output other index or all $\dim P_k^{(n)}$ indices.
  \item The output in Section \ref{code} lists the degrees of freedom at each one sub-simplex
      (0=vertex, 1=edge, and so on.)  At the end of each level of simplex,  the number of 
    sub-simplex and the subtotal of degrees of freedom are listed.
    At the end,  the dimension of $P_k^{(n)}$ and the number of degrees of freedom are listed.
    They match each other.
  \item By changing \verb+c1+, we can list each index as soon as it is assigned in to an index set.
  \item By changing \verb+c2+, we can list an index on high-dimension sub-simplex if it is assigned to
      a low-dimensional sub-simplex.   This would help us to find overlapping structure.
       This change requires corresponding change on \verb+c7+.
  
  \item By changing \verb+c3+, \verb+c4+ and \verb+c5+, we can output all 
         indices of one particular group. 
   
\end{enumerate}

\appendix
\section{The code computing dof of $C^m$-$P_k{(n)}$ and its output}\label{code}

A Fortran computer code computes the index set of nodal basis functions of $C^m$-$P_k{(n)}$
   finite elements, for any space dimension $n$.

\baselineskip 12pt
\begin{verbatim}

c0    nDCmPk(n,m,k1) for C^m-P_{2^n m+1+k1}^{(n)} element
      call nDCmPk(3,3,2)  
      end  

      subroutine nDCmPk(n,m,k1) 
      integer ix(4000000,10),ii(12,120,30,10),iz(8,120,10) 
      k=m*2**n+1+k1
      idim=1
      do i=1,n
       idim=idim*(k+i)/i
      enddo  
      if(idim.gt.4000000) stop 'inc ix dim'
      call izindex(n,iz)  
      call baryc(n,k,idim,ix) 
       
      do i=2,7
       iz(i,1,1)=0
      enddo 
      do i0=0,n 
       is=0
       do j=0,i0 
         is=max(is,iz(8,1,i0+1))
       enddo
       if(is.gt.30) stop 'inc ii in nDCmPk' 
       m1=m*2*(n-1-i0)+1
       if(i0.eq.n) m1=1  
       do j2=1,is
        do j1=1,m1 
         ii(1,j1,j2,1+i0)=0 
        enddo 
       enddo 
      enddo  
 23   format(' (',3(i2,','),i2')',6i3) 
 33   format('simplex', i2,' derivative',i2,' dof ',i7,'  sum=',i8) 
 43   format('level  ', i2,' #simplex  ',i2,' dofs',i7,' total',i8) 
 53   format('(n m k_1)=',3i2,',dim P_{',i3,'}=',i8,'C^m-P_k^n=',i8) 

      do il=0,n 
       call subs(n,m,k,idim,il,ix,ii,iz)  
      enddo  
      itl=0
      do i0=0,n 
       ic=0
       do i1=0,m*2**(n-i0-1) 
        is=0
        do i2=1,1 
         do i=1,idim   
         if(((ix(i,n+2).eq.i0+1).and.(ix(i,n+3).eq.i1+1))
     >      .and. (ix(i,n+4).eq.i2) ) then 
          is=is+1
c1         write(6,23) (ix(i,j),j=1,n+1),ic
         endif
         enddo
        enddo
        ic=ic+is
        write(6,33) i0,i1, is,ic 
c2      if(i0.eq.3) print*, ' overlap ', (iz(j,i1,10),j=1,i0)
       enddo
       write(6,43) i0,iz(8,1,i0+1), ic,ic*iz(8,1,i0+1)
       itl=itl+ic*iz(8,1,i0+1)
      enddo
      write(6,53) n,m,k1,k, idim, itl 
      end
    
      subroutine indexing(n,k,ix,i1,i2,i3,i4,i5,i6,j,i)
      integer ix(4000000,10),id(10)
      j=j+1
      if(i.eq.j) then
       id(1)=i1
       id(2)=i2
       id(3)=i3
       id(4)=i4
       id(5)=i5
       id(6)=i6
       ix(i,1+n)=0
       do l=1,n
        ix(i,l)=id(l)
        ix(i,1+n)=ix(i,1+n)+id(l)
       enddo
       ix(i,1+n)=k-ix(i,1+n)
       ix(i,2+n)=0
      endif
      end   
      
      subroutine baryc(n,k,idim,ix,iz) 
      integer ix(4000000,10) 
      do i=1,idim 
       j=0
       do i1=0,k
        do i2=0,k-i1
        if(n.eq.2)then
         call indexing(n,k,ix,i1,i2,i3,i4,i5,i6,j,i)
        else 
         do i3=0,k-i1-i2
         if(n.eq.3)then
          call indexing(n,k,ix,i1,i2,i3,i4,i5,i6,j,i)
         else 
          do i4=0,k-i1-i2-i3
          if(n.eq.4)then
           call indexing(n,k,ix,i1,i2,i3,i4,i5,i6,j,i)
          else 
           do i5=0,k-i1-i2-i3-i4
           if(n.eq.5)then
            call indexing(n,k,ix,i1,i2,i3,i4,i5,i6,j,i)
           else 
            do i6=0,k-i1-i2-i3-i4-i5
            call indexing(n,k,ix,i1,i2,i3,i4,i5,i6,j,i)
            enddo
           endif
           enddo
          endif
          enddo
         endif
         enddo
        endif
        enddo 
       enddo 
      enddo  
      end 
 
      subroutine dof(n,m,k,il,idim,ix,ii,iz,i1,i2,i3,i4,i5,i6)
      integer ix(4000000,10),ii(12,120,30,10),iz(8,120,10) 
      integer id(10)  
      id(1)=i1+1
      id(2)=i2+1
      id(3)=i3+1
      id(4)=i4+1
      id(5)=i5+1
      id(6)=i6+1 
      do kd=0,m*2**(n-il-1) 
       do i=1,idim 
c7      if((ix(i,n+2).eq.0).or. (il.eq.3) )
        if (ix(i,n+2).eq.0)  
     >     call ixy(n,m,k,il,kd,ix,ii,iz,i,id)
       enddo
      enddo
      end   
 
      subroutine subs(n,m,k,idim,il,ix,ii,iz)  
      integer ix(4000000,10),ii(12,120,30,10),iz(8,120,10)   
      do i1=0,n
      if(il.eq.0)then
       call dof(n,m,k,il,idim,ix,ii,iz,i1,i2,i3,i4,i5,i6)
       else
       do i2=i1+1,n
       if(il.eq.1)then
        call dof(n,m,k,il,idim,ix,ii,iz,i1,i2,i3,i4,i5,i6)
        else 
        do i3=i2+1,n
        if(il.eq.2)then
         call dof(n,m,k,il,idim,ix,ii,iz,i1,i2,i3,i4,i5,i6)
         else 
         do i4=i3+1,n
         if(il.eq.3)then
         call dof(n,m,k,il,idim,ix,ii,iz,i1,i2,i3,i4,i5,i6)
         else 
          do i5=i4+1,n
          if(il.eq.4)then
          call dof(n,m,k,il,idim,ix,ii,iz,i1,i2,i3,i4,i5,i6)
          else 
           do i6=i5+1,n 
           call dof(n,m,k,il,idim,ix,ii,iz,i1,i2,i3,i4,i5,i6) 
           enddo
          endif
          enddo
         endif
         enddo
        endif
        enddo 
       endif
       enddo  
      endif
      enddo 
      end 

      subroutine izs(j,j0,iz,i1,i2,i3,i4,i5,i6)  
      integer iz(8,120,10) 
      j=j+1 
      if(j.gt.120) stop ' inc j in izs '
      iz(1,j,j0)=i1+1
      iz(2,j,j0)=i2+1
      iz(3,j,j0)=i3+1
      iz(4,j,j0)=i4+1
      iz(5,j,j0)=i5+1
      iz(6,j,j0)=i6+1
      end 
      
      subroutine izindex(n,iz)  
      integer iz(8,120,10)   
      do j0=1,6
       j=0
       do i1=0,n 
        if(j0.eq.1) then
         call izs(j,j0,iz,i1,i2,i3,i4,i5,i6)
        else 
         do i2=i1+1,n 
          if(j0.eq.2) then
           call izs(j,j0,iz,i1,i2,i3,i4,i5,i6)
          else 
           do i3=i2+1,n 
            if(j0.eq.3) then
             call izs(j,j0,iz,i1,i2,i3,i4,i5,i6)
            else 
             do i4=i3+1,n 
              if(j0.eq.4) then
               call izs(j,j0,iz,i1,i2,i3,i4,i5,i6)
              else 
               do i5=i4+1,n 
                if(j0.eq.5) then
                 call izs(j,j0,iz,i1,i2,i3,i4,i5,i6)
                else 
                 do i6=i5+1,n 
                  call izs(j,j0,iz,i1,i2,i3,i4,i5,i6)
                 enddo
                endif
               enddo
              endif
             enddo
            endif
           enddo
          endif
         enddo
        endif
       enddo 
       iz(8,1,j0)=j  
      enddo   
      end
  
      subroutine ixy(n,m,k,il,kd,ix,ii,iz,i,id)
      integer ix(4000000,10),ii(12,120,30,10),iz(8,120,10) 
      integer id(10)  
      is=0
      do l=1,il+1
        is=is+ ix(i,id(l))
      enddo 
      if(is.eq.k-kd) then 
       do j=1,iz(8,1,il+1) 
        is=0
        do l=1,il+1
         is=is+abs(iz(l,j,il+1)-id(l))
        enddo
        if(is.eq.0) goto 2
       enddo
       print*, ' not found ', il, (id(l),l=1,il+1)
       stop ' n-found'
 2     continue   
      if(ix(i,n+2).eq.3) then
       if(((il.eq.3).and.(kd.eq.3)).and.(j.eq.1)) then
c3         iz(ix(i,n+2),kd+1,10)=iz(ix(i,n+2), kd+1,10)+1
c4        write(6,4) i,(ix(i,l),l=1,n+4), il+1,1+kd,j, 
c5    >        iz(ix(i,n+2), kd+1,10)
 4       format(i8, 5i3, ' used', 3i4, ' try', 3i3, ' #', i5)
       endif
      else
       ii(1,j,1+kd,il+1)=ii(1,j,1+kd,il+1)+1
       ix(i,n+2)=il+1
       ix(i,n+3)=1+kd
       ix(i,n+4)=j
       ix(i,n+5)=ii(1,j,1+kd,il+1) 
       if(((il.eq.2).and.(j.eq.1)).and.(kd.ge.0))then
         iz(kd+2,1,1)=iz(kd+2,1,1)+1
         if(iz(kd+2,1,1).lt.2) then
c6       write(6,13) iz(kd+1,1,1), ii(1,j,1+kd,il+1), j, kd,il 
         write(6,12) (ix(i,l),l=1,n+4), (iz(l,j,il+1),l=1,il+1)
         endif
       endif
13     format(2i8,'   ',5i4, '    ',10i4)
12     format('check: ',5i4, '    ',10i4)
      endif
      endif
      end   
\end{verbatim}

The output for $C^3$-$P_{3(4)+1+2}^{(3)}$:

\begin{verbatim}
check:    7   7  13   0   3       1   1   1   2   3
check:    6   6  14   1   3       2   1   1   2   3
check:    5   6  14   2   3       3   1   1   2   3
check:    4   6  14   3   3       4   1   1   2   3
simplex 0  derivative 0 dof       1  sum=       1
simplex 0  derivative 1 dof       3  sum=       4
simplex 0  derivative 2 dof       6  sum=      10
simplex 0  derivative 3 dof      10  sum=      20
simplex 0  derivative 4 dof      15  sum=      35
simplex 0  derivative 5 dof      21  sum=      56
simplex 0  derivative 6 dof      28  sum=      84
simplex 0  derivative 7 dof      36  sum=     120
simplex 0  derivative 8 dof      45  sum=     165
simplex 0  derivative 9 dof      55  sum=     220
simplex 0  derivative10 dof      66  sum=     286
simplex 0  derivative11 dof      78  sum=     364
simplex 0  derivative12 dof      91  sum=     455
level   0  #simplex   4 dofs    455 total    1820
simplex 1  derivative 0 dof       2  sum=       2
simplex 1  derivative 1 dof       6  sum=       8
simplex 1  derivative 2 dof      12  sum=      20
simplex 1  derivative 3 dof      20  sum=      40
simplex 1  derivative 4 dof      30  sum=      70
simplex 1  derivative 5 dof      42  sum=     112
simplex 1  derivative 6 dof      56  sum=     168
level   1  #simplex   6 dofs    168 total    1008
simplex 2  derivative 0 dof      28  sum=      28
simplex 2  derivative 1 dof      45  sum=      73
simplex 2  derivative 2 dof      63  sum=     136
simplex 2  derivative 3 dof      82  sum=     218
level   2  #simplex   4 dofs    218 total     872
simplex 3  derivative 0 dof     360  sum=     360
level   3  #simplex   1 dofs    360 total     360
 (n m k_1)= 3 3 2, dim P_{ 27}=    4060 C^m-P_k^n=    4060
\end{verbatim}

\end{document}